\documentclass[11pt]{article}

\usepackage{graphicx}              % to include figures
\usepackage{amsmath}               % great math stuff
\usepackage{amsfonts}              % for blackboard bold, etc
\usepackage{amsthm}                % better theorem environments
\usepackage{amssymb, pgf}
\usepackage{xcolor}
\usepackage{wrapfig}

\usepackage{tikz}
\usetikzlibrary{backgrounds}

\theoremstyle{plain}
\newtheorem{thm}{Theorem}[section]
\newtheorem{lem}[thm]{Lemma}

\newtheorem*{cor}{Corollary}

\theoremstyle{definition}
\newtheorem{defn}[thm]{Definition}
\newtheorem{alg}[thm]{Algorithm}

\newcommand{\p}{\mathcal{P}}

\newcommand{\B}{\mathcal{B}}
\newcommand{\W}{\text{width}}
\newcommand{\ord}{\text{order}}

\begin{document}

\nocite{*}

\title{Treewidth Bounds for Planar Graphs Using Three-Sided Brambles}
\author{Karen L. Collins\\
\small Dept. of Mathematics \& Computer Sci.\\
\small Wesleyan University\\
\small Middletown CT 06459\\
\small\tt kcollins@wesleyan.edu\
\and
Brett C. Smith \\
\small Dept. of Mathematics\\
\small Yale University\\
\small New Haven CT 06520 \\
\small\tt brett.c.smith@yale.edu
}

\maketitle

\begin{abstract}
Square grids play a pivotal role in Robertson and Seymour's work on graph minors as planar obstructions to small treewidth. 
We introduce a three-sided bramble in a plane graph called a net, which generalizes the standard bramble of crosses in a square grid. We then characterize any minimal cover of a net as a tree drawn in the plane. 

We use nets in an $O(n^3)$ time algorithm that computes both upper and lower bounds on the bramble number (hence treewidth) of any planar graph. Let $G$ be a planar graph, $BN(G)$ be its bramble number and $\lambda(G)$ be the largest order of any net in a subgraph of $G$. Our algorithm outputs a constant, $KB$, so that $\lambda(G)/4 \leq KB \leq BN(G)\leq 4KB \leq 4\lambda(G)$. 

Let $s(G)$ be the size of a side of the largest square grid minor of $G$. Smith (2015) has shown that $\lambda(G) \geq s(G)$. Our upper bound improves that of Grigoriev (2011) when $\lambda(G)\leq (5/4)s(G)$. We correct a lower bound of Bodlaender, Grigoriev and Koster (2008) to $s(G)/5$ (instead of $s(G)/4$) and thus the lower bound of $\lambda(G)/4$ on our approximation is an improvement.  

\vspace{12pt}

\end{abstract}

\section{Introduction}

The treewidth of a graph is a fundamental idea in Robertson and Seymour's pioneering work on graph minors. For any graph $G$, let $TW(G)$ be the treewidth of $G$ (see \cite{Robertson1986a}). The base case in the proof of the Graph Minor Theorem \cite{Robertson2004} relies on square grids for two reasons: (1) the $n \times n$ square grid has treewidth $n$, so the family of square grids has unbounded treewidth, and (2) each square grid is planar and hence has genus zero. Robertson and Seymour use their well-known Grid-Minor Theorem (also called the Excluded Grid Theorem) to start an induction on the genus of a graph. %The Grid-Minor Theorem says that for any $n\geq 0$ there exists some $k$ such that for any graph $G$, if $TW(G) \geq k$ then $G$ contains a minor isomorphic to the $n \times n$ square grid. 

%Topologically, a natural separating set in a plane graph is a cycle; the smallest such cycle in a simple graph has three vertices. While square grids are elegant examples for their regular structure, they are far from being minimal obstructions to small treewidth with respect to minors. We will define three-sided objects that obstruct small treewidth.

We replace squares with triangles in this role. In a graph, every 4-cycle contains a 3-cycle as a minor. In this sense, searching for a triangle is more general than searching for a square. 

We proceed by addressing the dual problem to treewidth, finding a graph's bramble number. Let $BN(G)$ be the bramble number of a graph $G$. Seymour and Thomas first introduced brambles (originally called screens) in \cite{Seymour1993} to get lower bounds on the treewidth of a graph. 

\begin{thm}[Seymour and Thomas] \label{duality}
For any graph $G$,  $$TW(G)=BN(G)-1.$$
\end{thm} 
\noindent Bellenbaum and Diestel have  a short proof of this result in \cite{Bellenbaum2002}. See also Reed \cite{Reed1997} for more on brambles. Bodlaender~\cite{Bodlaender1998} surveys various equivalent notions to treewidth.

Our method for bounding treewidth will be to define a special class of brambles called nets. Nets were introduced in Smith's thesis \cite{Smith2015}, where they  
were used  to construct two families of planar graphs, each containing a minor minimal obstruction to any treewidth. Nets are three-sided brambles in plane graphs. A formal definition of nets appears in Section~3.

%A bramble of $G$ is a collection of connected vertex subsets that pairwise meet or touch. A cover of bramble is a subset of vertices that intersects every element of the bramble. The order of a bramble is the size of its smallest cover. For any graph $G$, the bramble number, $BN(G)$, is the largest order of a bramble in $G$.

Nets can be thought of as a generalization of a natural bramble in a square grid called the {\it bramble of crosses}, described in \cite{Birmele2007}. Each cross contains vertices from one row and one column in the grid. Smith \cite{Smith2015} defines $n$-triangular grids for $n\geq 3$ (see figure~\ref{triangle}), and considers brambles whose elements meet all three sides of the triangle (rather than four sides, as in the bramble of crosses in the square grid). In particular, he proves that this bramble in the $n$-triangular grid has order $n$. Thus, triangular grids are a natural three-sided analogue to square grids. In this context, a bramble of crosses becomes a bramble of trees, each with a unique root and three branches. This three-sided bramble in a triangular grid provides a canonical example of a net in a plane graph.

It is important to note that the $n$-triangular grid does not contain an $n \times n$ square grid. The largest square grid minor has side-length less than or equal to $(n+1)/\sqrt{2}$ because, by counting vertices,  there are  $\left(^{n+1}_{\;\;\; 2}\right)$ vertices in an $n$-triangular grid and $m^2$ vertices in an $m \times m$ square grid. Thus, the bramble number of a triangular grid is larger than the bramble number of any of its square grid minors. These examples motivate us to improve lower bounds on treewidth for planar graphs found using square grids by finding high order three-sided nets.

%A net  is a three-sided bramble of a plane graph. The three sides of a net are called the frame. %The graph in Figure~\ref{simplenet} is given a frame with sides $\{a,b\}$, $\{b,c\}$ and $\{c,d,a\}$. This frame defines a net whose elements are connected subsets of vertices, each of which intersects all three sides. The minimal elements of the net defined this way are $\{a, b\}$, $\{b, c\}$, $\{a,c,d\}$ and $\{b, d, e\}$. 
We define $\lambda(G)$ to be the largest order of any net in a subgraph of $G$.  Let $s(G)$ be the size of the side of the largest square grid minor in $G$.  Smith~\cite{Smith2015} has shown that $s(G)\leq \lambda(G)$, and in particular when $G$ is an $s\times s$ square grid, then $\lambda(G)=s$. Chekuri and Chuzhoy (2016) show that there is a polynomial relationship between the treewidth and square grid minor size of a graph \cite{Chekuri2016}.  Since $\lambda(G)\geq s(G)$, their results demonstrate a polynomial relationship between the treewidth and the net order of a graph.

Smith \cite{Smith2015} has a polynomial time algorithm to compute the minimum cover of a net, giving a lower bound for $BN(G)$.  In Section~4, we present a faster such algorithm. Given a net of a planar graph, $G$, we construct an $O(n^2)$ time algorithm, Net-Alg, whose output is a minimum cover.
We characterize a cover as a tree drawn in the plane which may go through the faces, edges or vertices of $G$. Net-Alg proceeds by the shortest path algorithm from Henzinger, et al. \cite{Henzinger1997} and inspiration from Dreyfus \cite{Dreyfus1971}
to find a Steiner tree that meets all three sides of the net.

%Think of a plane graph as being contained in a large disc. We start by giving it a 3-frame. We cover it by a tree that meets three sides. The tree could be drawn in the disc through the face areas of the disc. So we add vertices to the graph to represent these routes. To find the minimum cover, we find the minimum Steiner tree of the larger graph and then intersect it with the original $G$. 

%To find upper bounds, we show that there is a natural net in an $n\times n$ grid with bramble number $n$. Define $\lambda(G)$ to be the size of the largest net in $G$ and let $g(G)$ be Adapting their work is a natural test case for us to see if nets can achieve improved bounds. These are recorded in Theorem~\ref{BT-Alg}, see below.

We use nets to replace square grids in the work of Bodlaender, Grigoriev and Koster \cite{Bodlaender2008}. They  construct a rooted-search-tree algorithm to find lower bounds on the treewidth of a graph.  
In particular, their algorithm finds a square grid minor of size $LB_2$. The claim is that $LB_2\geq s(G)/4$; however as we will show in Section~6, the proof shows that $LB_2\geq s(G)/5$ only. 
Grigoriev \cite{Grigoriev2011} writes a new algorithm, using ideas from \cite{Bodlaender2008}, to construct a tree decomposition of a graph and thus gets an upper bound on treewidth. We adapt ideas from both \cite{Bodlaender2008} and \cite{Grigoriev2011} in our rooted-search tree algorithm, BT-Alg. Our algorithm uses nets to achieve both lower and upper bounds on bramble number, as seen in the following theorem.

\vspace{.08in}

\noindent {\bf Theorem 5.7.} 
{\it Let $G$ be a planar graph. Then BT-Alg computes $KB$ in $O(|G|^3)$ time, and }
$$\frac{\lambda(G)}{4} \leq KB \leq BN (G) \leq 4KB \leq 4\lambda(G).$$\vspace{.08in}

%\noindent {\bf Theorem~5.5.} {\it There is a rooted-search-tree algorithm, BT-Alg, that inputs a planar graph $G$ and outputs a bramble and a tree-decomposition of $G$. Let the bramble size be $KB$. Then the bramble number has size at most $4KB$ and therefore} 
%\[BN(G)\geq KB\geq \frac{\lambda(G)}{4}\] \vspace{.08in}

%\begin{thm} \label{BT-Alg} There is a rooted-search-tree algorithm, BT-Alg, that inputs a planar graph $G$ and outputs a bramble and a tree-decomposition of $G$. Let the bramble size be $KB$. Then the tree-decomposition has size at most $4KB-1$ and therefore 
%\[TW(G)= BN(G)-1\geq KB-1\geq \frac{\lambda(G)}{4}-1\]
% \end{thm}

\noindent The proof appears in Section~5. Theorem~\ref{lower-and-upper} improves the upper bound of $5s(G)-6$ from \cite{Grigoriev2011} whenever $\lambda\leq 5s(G)/4$. Since $\lambda(G)\geq s(G)$, with the correction in Section~6, our lower bound is better than \cite{Bodlaender2008}. 
 
The outline of the paper is as follows. In Section~2 we do some preliminaries. In Section~3 we define nets and give an alternative characterization of a cover of net. In Section~4 we construct Net-Alg. In Section~5 we construct BT-Alg and prove Theorem~\ref{lower-and-upper}. In Section~6 we correct the lower bound in Bodlaender et al.  We make some concluding remarks in Section~7.

\section{Preliminaries}

In this paper, every graph will be simple; for a graph $G$, we let $V(G)$ denote its vertex set and $E(G)$ denote its edge set. For each vertex, $u \in V(G)$, we let $N_G(u)$ denote its neighborhood, the set of vertices adjacent to $u$ in $G$. For any subset of vertices, $U \subseteq V(G)$, let $G[U]$ denote the subgraph of $G$ induced by $U$. If $G[U]$ is connected, then we say the vertex set, $U$, is \textbf{connected} in $G$. We denote the induced subgraph, $G[V(G) \backslash U]$, by $G - U$. We will refer to the power set of vertices in $G$ as $\mathcal{P}(V(G))$. If $T$ is a tree and $s, t \in V(T)$, then let $sTt$ denote the unique path from $s$ to $t$ in $T$. Similarly, if $W = (u_0, ..., u_n)$ is a walk in a graph, then for any $0 \leq i \leq j \leq n$, let $u_iWu_j$ denote the subwalk from $u_i$ to $u_j$ in $W$. If $W$ is a closed walk, let $(u_jWu_n, u_1Wu_i)$ denote the concatenated walk, $(u_j, u_{j+1}, ..., u_n, u_1, u_2,  ..., u_i)$.

%We will use the following definitions. 
\begin{defn}
\label{tree_decomp}
Let $G$ be a graph, let $T$ be a tree and let $\beta: V(T) \rightarrow \p(V(G))$. The pair $(T, \beta)$ is a \textbf{tree decomposition} of $G$ if and only if
\begin{enumerate}
	\item if $u \in V(G)$, then there is some $t \in V(T)$ so that $u \in \beta(t)$, 
	\item if $uv \in E(T)$, then there is some $t \in V(T)$ such that $u,v \in \beta(t)$, AND
	\item if $t_1, t_2, t_3 \in V(T)$ and $t_2\in V(t_1Tt_3)$, then $\beta(t_1) \cap \beta(t_3) \subseteq \beta(t_2)$.
\end{enumerate}
\end{defn}

 The \textbf{width} of a tree decomposition is defined as $\W(T, \beta):=\max\{|\beta(t)|: t \in V(T)\} - 1$ and the \textbf{treewidth} of a graph $G$ is 
$$TW(G):= \min\{\W(T, \beta): (T, \beta) \ \text{is a tree decomposition of} \ G\}.$$

Using ideas from Reed in \cite{Reed1997}, we present an equivalent definition for a tree decomposition by looking at the inverse mapping:
\begin{align*}
	\beta^{-1}: V(G) &\rightarrow \mathcal{P}(V(T)) \\
			u &\mapsto \{t \in V(T) : u \in \beta(t)\}.
\end{align*}
\begin{lem}\label{alttw}
Let $G$ be a graph, $T$ a tree and $\beta: V(T) \rightarrow \mathcal{P}(V(G))$. Then $(T, \beta)$ is a tree decomposition of $G$ just in case both
\begin{enumerate}
	\item if $u \in V(G)$, then $\beta^{-1}(u)$ is nonempty and connected in $T$, AND
	\item if $uv \in E(G)$, then $\beta^{-1}(u) \cap \beta^{-1}(v) \neq \varnothing$.
\end{enumerate}
\end{lem}
The proof of lemma~\ref{alttw} follows directly from the definition of a tree decomposition. 
Now we will give the precise definition of a bramble. 
\begin{defn}
Let $G$ be a graph. For two subsets of vertices, $A, B \subseteq V(G)$, we say that $A$ and $B$ \textbf{touch} if either
\begin{enumerate}
	\item $A\cap B \neq \varnothing$, OR
	\item there is some $uv \in E(G)$ such that $u \in A$ and $v \in B$.
\end{enumerate}

A finite collection, $\B = \{B_i \subseteq V(G)\}_{i \in [n]}$, of connected, mutually touching vertex sets is called a \textbf{bramble} in $G$. A set, $C \subseteq V(G)$ \textbf{covers} $\B$ if $C$ nontrivially intersects each vertex set in the bramble. The \textbf{order} of a bramble $\B$ is defined as $\ord(\B):= \min\{|C| : C \text{ covers } \B\}$, and the \textbf{bramble number} of a graph, $G$, is 
	$$BN(G):= \max\{\ord(\B): \B \text{ is a bramble in } G\}.$$
\end{defn}

For a simple example of a bramble in a graph, we can take the collection of vertices (as singleton sets) from a clique. The order of this bramble is the number of vertices in the clique. In particular, the bramble number of any graph is bounded below by its clique number. However, this example does not take advantage of the ability to have intersection among sets. Allowing sets in a bramble to intersect gives a much more robust class of obstructions to low treewidth.
%We will see that allowing this intersection leads to a surprisingly rich family of structures. 
%The upshot of Theorem~\ref{duality} is that high-order brambles can provide very good lower bounds on the treewidth of a graph.  

When talking about graph embeddings, we follow the conventions of West \cite{West1996}. In particular, we refer to a graph embedded in the plane as a \textbf{plane graph}, and for a plane graph, $G$, a \textbf{face}, $F$, of $G$ is a maximal region in the plane such that for any two points in $F$, there is a curve avoiding $G$ connecting those points. For computational reasons, we point out that in any plane graph, if $v$ is the number of vertices, $e$ is the number of edges and $f$ is the number of faces, then $e \leq 3v - 6$ and $f \leq 2v-4$. Therefore, any computation that iterates on vertices, edges and faces of a graph still runs in $O(v)$-time. With this in mind, for any planar graph we let $|G| = |V(G)|$.

\section{Introducing nets and characterizing net covers}

We are ready to define a natural family of three-sided brambles that occur in plane graphs, generalizing the bramble of crosses in a square grid graph described in the introduction. 

\begin{defn} Let $G$ be a connected plane graph and let $W = (u_0, u_1,  ...,  u_n)$ be a closed walk peripheral to the unbounded face of $G$, so $u_0=u_n$.  For any choice of $j$ and $k$ so that $0 \leq j \leq k \leq n$, we call the triple, $\mathcal{F} = (W, j, k)$, a \textbf{3-frame} of $G$.
\end{defn}

A 3-frame decomposes $W$ into three subwalks with overlapping endpoints. %See Figure~\ref{triangle}.  $u_0Wu_j$, $u_jWu_k$ and $u_kWu_n$.

\begin{defn} Given a 3-frame, $\mathcal{F} = (W, j, k)$, the three subwalks, $u_0Wu_j$, $u_jWu_k$ and $u_kWu_n$, are called the \textbf{sides} of $\mathcal{F}$. 
\end{defn}

Throughout the rest of the paper we refer to the vertex sets of the sides of the 3-frame by colors:
$$blue = \{u_0, ..., u_j\}; \ \ \ red = \{u_j, ..., u_k\}; \ \ \ yellow = \{u_k, ..., u_n\}$$

\begin{defn} Given a 3-frame, $\mathcal{F}$, we call $X\subseteq V(G)$ an \textbf{$\mathcal{F}$-vine} if $X$ is connected and contains at least one vertex in each side of $\mathcal{F}$. 
\end{defn}

That is, an $\mathcal{F}$-vine has at least one $blue$, one $red$ and one $yellow$ vertex.

\begin{defn} Given a 3-frame, $\mathcal{F}$, we define the \textbf{$\mathcal{F}$-net} of $G$ as  the collection of all $\mathcal{F}$-vines. We denote the $\mathcal{F}$-net of $G$ by $\mathbf{N}(G,\mathcal{F})$.
% =\{X \subseteq V(G) : X \text{ is an } \mathcal{F}\text{-vine}\}.$
\end{defn}

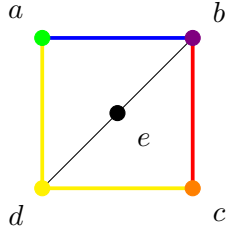
\begin{figure}[h]
\begin{center}
\begin{tikzpicture}
	\draw [line width=0.5mm, blue] (0,2) -- (2,2);
	\draw [line width=0.5mm, red] (2,2) -- (2,0);
	\draw [line width=0.5mm, yellow] (2,0) -- (0,0) -- (0,2);
	\draw (0,0) -- (2,2);
	\draw [fill, yellow] (0,0) circle [radius=0.1];
	\draw [fill, green] (0,2) circle [radius=0.1];
	\draw [fill, orange] (2,0) circle [radius=0.1];
	\draw [fill, violet] (2,2) circle [radius=0.1];
	\draw [fill] (1,1) circle [radius=0.1];
	\node at (-.35,-.35) {$d$};
	\node at (2.35,-.35) {$c$};
	\node at (2.35,2.35) {$b$};
	\node at (-.35,2.35) {$a$};
	\node at (1.35,0.65) {$e$};
\end{tikzpicture}
\caption{A net with sides $\{a,b\}$, $\{b,c\}$ and $\{c,d,a\}$}\label{simplenet}
\end{center}
\end{figure}

The graph in Figure~\ref{simplenet} is given a frame with sides $\{a,b\}$, $\{b,c\}$ and $\{c,d,a\}$. This frame defines a net whose elements are connected subsets of vertices, each of which intersects all three sides. The minimal elements of this net are $\{a, b\}$, $\{b, c\}$, $\{a,c,d\}$ and $\{b, d, e\}$. Figure~\ref{triangle} shows a 3-frame of a 6-triangular grid. 

Note that the vertex sets $blue$, $red$ and $yellow$ are not disjoint, sharing at least one vertex between pairs. In the case where $W$ is a cycle, each pair of colors intersect on exactly one vertex, but both $W$ and our choice of $j$ and $k$ may cause more intersection between the sides.

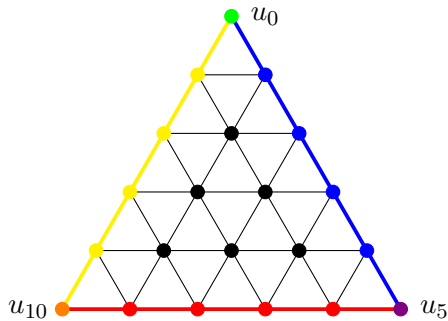
\begin{figure}[h]
\begin{center}
\begin{tikzpicture}[scale=.9]
	\node at (3, 4.33) {$u_0$};
	\node at (-0.5, 0) {$u_{10}$};
	\node at (5.5, 0) {$u_5$};
	\draw [line width=0.5mm, red ] (0, 0) -- (5, 0) ;
	\draw [line width=0.5mm, blue ] (5, 0) -- (2.5, 4.33);
	\draw [line width=0.5mm, yellow ](0, 0) -- (2.5, 4.33);
	\draw (2, 3.464) -- (3, 3.464);
	\draw (1, 0) -- (3, 3.464);
	\draw (4, 0) -- (2, 3.464);
	\draw (1.5, 2.598) -- (3.5, 2.598);
	\draw (2, 0) -- (3.5, 2.598);
	\draw (3, 0) -- (1.5, 2.598);
	\draw (1, 1.732) -- (4, 1.732);
	\draw (3, 0) -- (4, 1.732);
	\draw (2, 0) -- (1, 1.732);
	\draw (0.5, 0.866) -- (4.5, 0.866);
	\draw (4, 0) -- (4.5, 0.866);
	\draw (1, 0) -- (0.5, 0.866);
	\draw [fill, orange] (0,0) circle [radius=0.1];
	\draw [fill, red] (1,0) circle [radius=0.1];
	\draw [fill, red] (2,0) circle [radius=0.1];
	\draw [fill, red] (3,0) circle [radius=0.1];
	\draw [fill, red] (4,0) circle [radius=0.1];
	\draw [fill, violet] (5,0) circle [radius=0.1];
	\draw [fill, yellow] (0.5, 0.866) circle [radius=0.1];
	\draw [fill] (1.5, 0.866) circle [radius=0.1];
	\draw [fill] (2.5, 0.866) circle [radius=0.1];
	\draw [fill] (3.5, 0.866) circle [radius=0.1];
	\draw [fill, blue] (4.5, 0.866) circle [radius=0.1];
	\draw [fill, yellow] (1, 1.732) circle [radius=0.1];
	\draw [fill] (2, 1.732) circle [radius=0.1];
	\draw [fill] (3, 1.732) circle [radius=0.1];
	\draw [fill, blue] (4, 1.732) circle [radius=0.1];
	\draw [fill, yellow] (1.5, 2.598) circle [radius=0.1];
	\draw [fill] (2.5, 2.598) circle [radius=0.1];
	\draw [fill, blue] (3.5, 2.598) circle [radius=0.1];
	\draw [fill, yellow] (2, 3.464) circle [radius=0.1];
	\draw [fill, blue] (3, 3.464) circle [radius=0.1];
	\draw [fill, green] (2.5, 4.33) circle [radius=0.1];
\end{tikzpicture}
\caption{A 3-frame, $(W, 5, 10)$, in a 6-triangular grid.}\label{triangle}
\end{center}
\end{figure}

\begin{lem} \label{net_bramble}
	Let $G$ be a connected plane graph with a 3-frame $\mathcal{F}$. If $X_1, X_2 \in \mathbf{N}(G,\mathcal{F})$, then $X_1 \cap X_2 \neq \varnothing$.
\end{lem}
\begin{proof} For the sake of contradiction, suppose there are $X_1, X_2 \in \mathbf{N}(G,\mathcal{F})$ such that $X_1 \cap X_2 = \varnothing$. Define a plane graph obtained from $G$ as follows. First, add three new vertices, $x_{b}$, $x_{r}$ and $x_{y}$, to the unbounded face of $G$ so that $x_b$ is adjacent to each vertex in $blue$, $x_r$ is adjacent to each vertex in $red$ and $x_y$ is adjacent to each vertex in $yellow$. Moreover, draw these edges so that $x_b$, $x_r$ and $x_y$ remain peripheral to the unbounded face of the graph. Then add another vertex $x_0$ to the unbounded face and draw edges from $x_0$ to $x_b$, $x_r$ and $x_y$ in such a way that preserves our proper embedding.  The vertex sets $\{x_b\}$, $\{x_r\}$ $\{x_y\}$, $\{x_0\}$, $X_1$ and $X_2$ induce a minor isomorphic to $K_{3,3}$. We have a proper planar embedding of this graph, contradicting Wagner's theorem. Therefore, $X_1 \cap X_2 \neq \varnothing$.
\end{proof}

The proof of the lemma demonstrates how the definition of a net takes advantage of the topology of the plane to guarantee intersection between any two vines. We will continue to use properties of our embedding to understand the order of these brambles.

In order to understand what a minimum size cover of a net looks like, we consider what would topologically prevent a vine touching all three sides of the 3-frame. As we saw in the proof of lemma~\ref{net_bramble}, any two vines have non-trivial intersection. That is, each vine in a net is itself a cover of the bramble.

On the other hand, a sparse plane graph may not take full advantage of the space to find vines with few vertices. See Figure~\ref{fig3}. For example, consider a cycle on 15 vertices, $C_{15} = (c_0, c_1, ..., c_{15})$.  We could evenly divide the cycle into a 3-frame: $(C_{15}, 5, 10)$, and any $(C_{15}, 5, 10)$-vine would use at least five vertices. However, a set of two vertices, $\{c_{5}, c_{10}\}$, covers the $(C_{15}, 5, 10)$-net. We can use the embedding to understand why this set covers the bramble by adding an edge between $c_{5}$ and $c_{10}$. This additional edge would make $\{c_{5}, c_{10}\}$ a $(C_{15}, 5, 10)$-vine, and any embedding of $C_{15}$ would afford us the space to make such an edge. 

To account for these latent vines, we need to pay attention to what connections are possible through faces of the embedding. We now define a new plane graph obtained from a plane graph that inherits connectivity from the bounded faces in our embedding.

\begin{figure}[h]
\begin{center}
\begin{tikzpicture}[scale=.8]
	\node at (3, 4.33) {$c_0$};
	\node at (-0.6, 0) {$c_{10}$};
	\node at (5.6, 0) {$c_5$};
	\draw [line width=0.5mm, red ] (0, 0) -- (5, 0) ;
	\draw [line width=0.5mm, blue ] (5, 0) -- (2.5, 4.33);
	\draw [line width=0.5mm, yellow ](0, 0) -- (2.5, 4.33);
	\draw [dashed] (0,0) .. controls (2.5, 1) .. (5, 0);
	\draw [fill, orange] (0,0) circle [radius=0.2];
	\draw [fill, red] (1,0) circle [radius=0.1];
	\draw [fill, red] (2,0) circle [radius=0.1];
	\draw [fill, red] (3,0) circle [radius=0.1];
	\draw [fill, red] (4,0) circle [radius=0.1];
	\draw [fill, violet] (5,0) circle [radius=0.2];
	\draw [fill, yellow] (0.5, 0.866) circle [radius=0.1];
	\draw [fill, blue] (4.5, 0.866) circle [radius=0.1];
	\draw [fill, yellow] (1, 1.732) circle [radius=0.1];
	\draw [fill, blue] (4, 1.732) circle [radius=0.1];
	\draw [fill, yellow] (1.5, 2.598) circle [radius=0.1];
	\draw [fill, blue] (3.5, 2.598) circle [radius=0.1];
	\draw [fill, yellow] (2, 3.464) circle [radius=0.1];
	\draw [fill, blue] (3, 3.464) circle [radius=0.1];
	\draw [fill, green] (2.5, 4.33) circle [radius=0.1];
\end{tikzpicture}
\caption{$\{c_5, c_{10}\}$ covers the net framed by $(C_{15}, 5, 10)$.} \label{fig3}.
\end{center}
\end{figure}
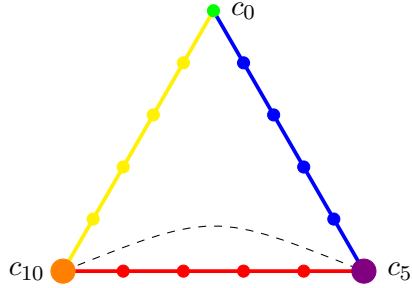 

\begin{defn}\label{face_alg} \ 
Let $G$ be a connected plane graph. For each bounded face, $f \in F(G)$: 
\begin{enumerate}
	\item Create a vertex $v_f$ inside the face, $f$. 
	\item Add an edge between $v_f$ and each vertex peripheral to $f$.
\end{enumerate}
We call the resulting plane graph the \textbf{face graph} of $G$, denoted $\widehat{G}$.
\end{defn}

Since each edge of $G$ is contained in at most two faces of $G$, a greedy search on edges would find all faces in $O(|G|)$ time. Notice that $\mathcal{F}$ is also a 3-frame of $\widehat{G}$, so $\mathbf{N}(G,\mathcal{F}) \subseteq \mathbf{N}({\widehat{G}}, \mathcal{F})$. The face graph is something like a combination of a plane graph with its dual and will give us a useful property concerning the connectedness of the graph. 

The following definition is inspired by a similar definition used by Robertson and Seymour in \cite{Robertson1986}.

\begin{defn} Let $G$ be a connected plane graph and let $W = (u_0, u_1, ..., u_n)$ be a closed walk peripheral to $G$. For any quadruple, $0 \leq a \leq b \leq c \leq d \leq n$, we say $(u_a,u_c)$ and $(u_b,u_d)$ \textbf{cross} in $W$. 
\end{defn}

The Jordan Curve Theorem implies for any two paths in a connected plane graph, if the endpoints of one cross the endpoints of another in the peripheral walk, then the two paths share a common vertex. We use this fact in the following useful characterization of separating sets in face graphs. 

\begin{lem} \label{face} Let $G$ be a connected plane graph, let $W$ be a closed walk peripheral to the unbounded face of $G$ and let $X \subseteq V(G)$. For any two vertices $u, v \in W\backslash X$, there is a $(u, v)$-path in $\widehat{G} - X$ if and only if there is no path in $G[X]$ whose endpoints cross $(u, v)$ in $W$. 
\end{lem}
\begin{proof} 
We prove the forward direction by assuming there is some $(u,v)$-path, say $P$, in $\widehat{G} - X$. If there was some $(x,y)$-path in $G[X]$ so that $(u, v)$ and $(x, y)$ cross in $W$, then it would necessarily share a vertex with $P$. This is impossible since $V(P)$ is disjoint with $X$.

For the reverse implication, restrict our embeddings of $G$ and $\widehat{G}$ to the surface obtained by throwing out the unbounded face of $G$ (which is also the unbounded face of $\widehat{G}$). If there is no path in $G[X]$ whose endpoints cross $(u, v)$ in $W$, then $u$ and $v$ are contained in a single face of $G[X]$. By the definition of a face, there is some polygonal $(u, v)$-curve, say $\gamma$, contained in that face. As we follow along this curve from $u$ to $v$, we obtain a finite multi-sequence, $M$, of vertices, edges and faces of $G$ that intersect $\gamma$. This multi-sequence starts with the vertex, $u$, and is followed either by a face or an edge of $G$. In fact, any consecutive pair in $M$ has of one of six forms: (vertex, face), (vertex, edge), (edge, face), (edge, vertex), (face, edge) or (face, vertex). We will use $M$ to find a $(u, v)$-walk in $\widehat{G} - X$. 

Obtain a multi-sequence, $M'$, of vertices from $\widehat{G} - X$ as follows. For each face, $f$, of $G$ in $M$, replace $f$ with $v_f \in V(\widehat{G})$. Clearly, $v_f \notin X$ since $X \subseteq V(G)$. For each edge, $e$, of $G$ in $M$, since $e$ is not an edge of $G[X]$, at least one of its endpoints is in $V(G)\backslash X$. Therefore, we can replace $e$ with one of its endpoints, $v_e$, that is not in $X$ (choosing at random if both endpoints are not in $X$). 

In search of a $(u, v)$-walk, we consider the six possible types of consecutive pairs of vertices in $M'$. In fact, since the edge relation is symmetric, it is enough to consider just the following three types:

	\textbf{Type 1: (vertex, face).} If $\gamma$ transitions from a vertex, $w \in V(G - X)$, to a face, $f$, of $G$, then $w$ must be on the boundary of $f$. By the definition of $\widehat{G}$, we know that $wv_f \in E(\widehat{G})$.
	
	\textbf{Type 2: (vertex, edge).} If $\gamma$ transitions from a vertex, $w \in V(G - X)$, to an edge, $e$, of $G$, then $w$ must be an endpoint of $e$. Thus, either $w = v_e$ or $wv_e \in E(G - X)$.
	
	\textbf{Type 3: (edge, face).} If $\gamma$ transitions from an edge, $e$, of $G$ to a face, $f$, of $G$, then $e$ is on the boundary of $f$ and so are both of its endpoints. Thus, $v_ev_f \in E(\widehat{G} - X)$.
	
	From the case analysis, we see that $M'$ must contain a $(u, v)$-walk in $\widehat{G} - X$ as a subsequence. Therefore, there is a $(u, v)$-path in $\widehat{G} - X$.
\end{proof}

We are ready to give an alternative characterization for a vertex set covering a net in a plane graph. 

\begin{thm} \label{min_cover} Let $G$ be a connected plane graph with a 3-frame $\mathcal{F}$ and let $C \subseteq V(G)$. Then $C$ covers $\mathbf{N}(G, \mathcal{F})$ if and only if there is some $Y \in \mathbf{N}(\widehat{G}, \mathcal{F})$ such that $Y \cap V(G) \subseteq C$. 
\end{thm}
\begin{proof} For the backward implication, let $Y \in \mathbf{N}(\widehat{G}, \mathcal{F})$. Consider any $X \in \mathbf{N}(G, \mathcal{F})$. Since $G$ is a subgraph of $\widehat{G}$, we know that $X \in \mathbf{N}(\widehat{G}, \mathcal{F})$, and lemma~\ref{net_bramble} implies that $X \cap Y \neq \varnothing$. Because $X \subseteq V(G)$, we have $X \cap (Y \cap V(G)) \neq \varnothing$. And since $X$ was chosen arbitrarily from $\mathbf{N}(G, \mathcal{F})$, we can conclude that $Y \cap V(G)$ covers $\mathbf{N}(G, \mathcal{F})$.

For the reverse implication, suppose $C$ covers $\mathbf{N}(G, \mathcal{F})$. Define $W = (u_0, u_1, ..., u_n)$ and $0 \leq j \leq k \leq n$ so that $\mathcal{F} = (W, j, k)$. Let $a$ be the maximum index in $\{0, ..., j\}$ such that there is a path, say $P_a$, from some vertex in $u_kWu_n$ to $u_a$ in $G - C$. Let $b$ be the maximum index in $\{j, ..., k\}$ such that there is a path, say $P_b$, from some vertex in $u_0Wu_j$ to $u_b$ in $G - C$, and let $c$ be the maximum index in $\{k, ..., n\}$ such that there is a path, say $P_c$, from a vertex in $u_jWu_k$ to $u_c$ in $G - C$. Since $C$ covers $\mathbf{N}(G, \mathcal{F})$, we see that $a<j$, $b<k$ and $c<n$, and $u_{a+1}, u_{b+1}, u_{c+1} \in C$.

We claim that by our choice of $a$ and $b$, there is no path in $G - C$ whose endpoints cross $(u_{a+1}, u_{b+1})$. For the sake of contradiction, suppose such a path, say $Q$, does exist. Let $u_s, u_t \in V(G) - C$ be the endpoints of $Q$, where $a+1 < s < b+1$ and $b+1 < t \leq n$ or $0 < t < a+1$. We consider the two possibilities for $s$. For example, see Figure~\ref{fig4}.

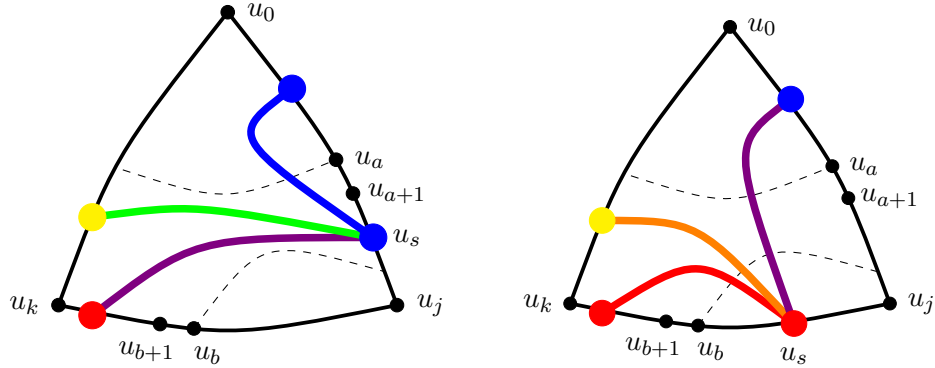
\begin{figure}[h]
\begin{center}
\begin{tikzpicture}[scale=.9]
	\draw [line width=0.5mm](0, 0) .. controls (2.5, -0.5) .. (5, 0) .. controls (4.2, 2.165) .. (2.5, 4.33) .. controls (0.8, 2.165) .. cycle;
	\draw [fill] (0,0) circle [radius=0.1];
	\draw [fill] (5,0) circle [radius=0.1];
	\draw [fill] (2.5, 4.33) circle [radius=0.1];
	\draw [fill, blue] (4.65, 1) circle [radius=0.2];
	\draw [fill, red] (0.5, -0.15) circle [radius=0.2];
	\draw [fill, yellow] (0.5, 1.3) circle [radius=0.2];
	\draw [fill, blue] (3.45, 3.2) circle [radius=0.2];
	\draw [fill] (2, -0.345) circle [radius=0.1];
	\draw [fill] (1.5, -0.28) circle [radius=0.1];
	\draw [fill] (4.1, 2.15) circle [radius=0.1];
	\draw [fill] (4.35, 1.65) circle [radius=0.1];
	\node at (3, 4.33) {$u_0$};
	\node at (5.5, 0) {$u_j$};
	\node at (-0.5, 0) {$u_k$};
	\node at (2.2, -0.745) {$u_b$};
	\node at (1.3, -0.68) {$u_{b+1}$};
	\node at (4.6, 2.15) {$u_a$};
	\node at (5, 1.65) {$u_{a+1}$};
	\node at (5.15, 1) {$u_s$};
	
	\begin{scope}[on background layer]
        	\draw[dashed] (0.95, 2) .. controls (2.5, 1.5) .. (4.1, 2.15);
	\draw[dashed] (2, -0.345) .. controls (3, 1) .. (4.8, 0.5);
	\draw[line width=1mm, violet] (4.65, 1) .. controls (2, 1) .. (0.5, -0.15);
	\draw[line width=1mm, green] (4.65, 1) .. controls (2, 1.5) .. (0.5, 1.3);
	\draw[line width=1mm, blue] (4.65, 1) .. controls (2.5, 2.5) .. (3.45, 3.2);
        \end{scope}
\end{tikzpicture}
\hspace{15pt}
\begin{tikzpicture}[scale=.85]
	\draw [line width=0.5mm](0, 0) .. controls (2.5, -0.5) .. (5, 0) .. controls (4.2, 2.165) .. (2.5, 4.33) .. controls (0.8, 2.165) .. cycle;
	\draw [fill] (0,0) circle [radius=0.1];
	\draw [fill] (5,0) circle [radius=0.1];
	\draw [fill] (2.5, 4.33) circle [radius=0.1];
	\draw [fill, red] (3.5, -0.313) circle [radius=0.2];
	\draw [fill, red] (0.5, -0.15) circle [radius=0.2];
	\draw [fill, yellow] (0.5, 1.3) circle [radius=0.2];
	\draw [fill, blue] (3.45, 3.2) circle [radius=0.2];
	\draw [fill] (2, -0.345) circle [radius=0.1];
	\draw [fill] (1.5, -0.28) circle [radius=0.1];
	\draw [fill] (4.1, 2.15) circle [radius=0.1];
	\draw [fill] (4.35, 1.65) circle [radius=0.1];
	\node at (3, 4.33) {$u_0$};
	\node at (5.5, 0) {$u_j$};
	\node at (-0.5, 0) {$u_k$};
	\node at (2.2, -0.745) {$u_b$};
	\node at (1.3, -0.68) {$u_{b+1}$};
	\node at (4.6, 2.15) {$u_a$};
	\node at (5, 1.65) {$u_{a+1}$};
	\node at (3.5, -0.813) {$u_s$};
	
	\begin{scope}[on background layer]
        	\draw[dashed] (0.95, 2) .. controls (2.5, 1.5) .. (4.1, 2.15);
	\draw[dashed] (2, -0.345) .. controls (3, 1) .. (4.8, 0.5);
	\draw[line width=1mm, red] (3.5, -0.313) .. controls (2, 0.8) .. (0.5, -0.15);
	\draw[line width=1mm, orange] (3.5, -0.313) .. controls (2, 1.3) .. (0.5, 1.3);
	\draw[line width=1mm, violet] (3.5, -0.313) .. controls (2.5, 2.5) .. (3.45, 3.2);
        \end{scope}
\end{tikzpicture}
\caption{All possible ways the endpoints of a path in $G - C$ might cross $(u_{a+1}, u_{b+1})$.} \label{fig4}
\end{center}
\end{figure}

\textbf{Case 1:} Suppose $a + 1 < s \leq j$. If $b+1 < t \leq k$, then $Q$ is a path from $u_0Wu_j$ to $u_t$ and we contradict the maximality of $b$. If $k < t \leq n$, then $Q$ is a path from $u_kWu_n$ to $u_s$ and we contradict the maximality of $a$. Finally, if $0 < t < a+1$, then the endpoints of $P_a$ cross the endpoints of $Q$ in $W$, so $P_a$ and $Q$ are subpaths of the same component of $G - C$. This component has a path from a vertex in $u_kWu_n$ to $u_s$, contradicting the maximality of $a$. Therefore, Case 1 leads to a contradiction.

\textbf{Case 2:} Suppose $j < s < b + 1$. If $b+1 < t \leq k$, then the endpoints of $P_b$ cross the endpoints of $Q$ in $W$.  Hence, $P_b$ and $Q$ are subpaths of the same component of $G - C$, and this component has a path from a vertex in $u_0Wu_j$ to $u_t$, contradicting the maximality of $b$. If $k < t \leq n$, then the endpoints of $P_b$ cross the endpoints of $Q$ in $W$, so $P_b$ and $Q$ are subpaths of the same component of $G - C$. But this component must contain vertices in all three sides of $\mathcal{F}$, so it contains an $\mathcal{F}$-vine. This contradicts the fact that $C$ covers $\mathbf{N}(G, \mathcal{F})$.  Finally, if $0 < t < a+1$, then the endpoints of $P_a$ cross the endpoints of $Q$ in $W$, so $P_a$ and $Q$ are subpaths of the same component of $G - C$. But this component must contain vertices in all three sides of $\mathcal{F}$, so it contains an $\mathcal{F}$-vine. This contradicts the fact that $C$ covers $\mathbf{N}(G, \mathcal{F})$. Therefore, case 2 leads to a contradiction.

Let $X = V(G) - C$. We have seen there is no path in $G[X]$ whose endpoints cross $(u_{a+1}, u_{b+1})$. Lemma~\ref{face} implies there is a $(u_{a+1}, u_{b+1})$-path in $\widehat{G} - X$. The same argument shows there is a $(u_{b+1}, u_{c+1})$-path in $\widehat{G} - X$. Therefore, $V(\widehat{G}) \backslash X$ contains an $\mathcal{F}$-vine, say $Y$, in $\widehat{G}$. By the definition of $X$, $Y \cap V(G) \subseteq C$, which completes the proof.
\end{proof}

\begin{cor} Let $G$ be a connected plane graph with at 3-frame $\mathcal{F}$. Then  
$$\emph{\ord}(\mathbf{N}(G, \mathcal{F})) = \min\{|Y \cap V(G)| : Y \in \mathbf{N}(\widehat{G}, \mathcal{F}) \}.$$
\end{cor}

Using the topology of a planar embedding, the problem of finding a cover of a net is equivalent to finding a connected subgraph of the face graph that intersects all three sides of the 3-frame. In the next section, we algorithmically minimize such a cover using shortest paths.

\section{Net-Alg: a minimum size cover of a net}

Let $G$ be a connected plane graph with a 3-frame $\mathcal{F}$. 
In order to develop an algorithm that finds a minimum order cover of a net, define the indicator function $\mathbf{1}_{G}$ on $V(\widehat{G})$ such that
$$\mathbf{1}_{G}(v) = \left\{\begin{array}{ l l }
			1, & v \in V(G) \\
			0, & v \in V(\widehat{G}) \backslash V(G). 
			\end{array}\right.$$
Let $G_{\leftrightarrow}$ denote a directed graph obtained from $G$ by replacing each edge with two directed edges of opposite orientation. Any directed path in $G_{\leftrightarrow}$ will then correspond to an undirected path of $G$ and vice versa. Moreover, two directed paths are vertex-disjoint in $G_{\leftrightarrow}$ just in case the corresponding paths in $G$ are vertex-disjoint.

From $\mathbf{1}_{G}$, we obtain a weight function, $\mathbf{1}_{G\leftrightarrow}$, on the edges of $\widehat{G}_{\leftrightarrow}$, where each edge weight is given by the weight of its terminal vertex under $\mathbf{1}_{G}$. We determine the weight of a directed path in $\widehat{G}_\leftrightarrow$ to be the sum of its edge weights. Then $\mathbf{1}_{G\leftrightarrow}$ gives us the following distance function on $\widehat{G}_\leftrightarrow$.
\begin{align*}
\delta_{\leftrightarrow}: V(\widehat{G}_\leftrightarrow) \times V(\widehat{G}_\leftrightarrow) &\rightarrow \{0, 1, 2, ...\}  \\
				(u, v) &\mapsto \min\left\{\sum_{e \in E(P)} \mathbf{1}_{G\leftrightarrow}(e) : P \ \text{is a $(u,v)$-path in} \ \widehat{G}_\leftrightarrow\right\}
\end{align*}

We use this distance function to define a specific structure possessed by any ``minimum weight'' $\mathcal{F}$-vine in $\widehat{G}$. This structure is essentially the three vertex case of the Steiner tree characterization given by Dreyfus and Wagner in \cite{Dreyfus1971}.

\begin{lem} \label{vine_cover} Let $G$ be a connected plane graph with a 3-frame $\mathcal{F}$. Suppose $Y \in \mathbf{N}(\widehat{G}, \mathcal{F})$ such that $|Y \cap V(G)|$ is minimum. Then $\widehat{G}_{\leftrightarrow}[Y]$ contains a rooted tree $T_{\leftrightarrow}$, where $T_\leftrightarrow$ is the union of three internally disjoint shortest paths (with distance given by $\delta_{\leftrightarrow}$), each starting at the root and terminating in one of the three sides of $\mathcal{F}$. 
\end{lem}
\begin{proof} By definition, $\widehat{G}[Y]$ contains a path with one endpoint in the $blue$ side of $\mathcal{F}$ and the other in the $red$ side. Let $P$ be such a path with $|V(P) \cap V(G)|$ minimum, and let $u_b$ and $u_r$ be the endpoints of $P$ in $blue$ and $red$, respectively. By definition, $Y$ also contains at least one vertex in the $yellow$ side of $\mathcal{F}$. Let $u_y \in Y$ be in $yellow$. Since $\widehat{G}[Y]$ is connected, there is a path in $\widehat{G}[Y]$ starting in $P$ and terminating at $u_y$. Let $Q$ be such a path minimizing $|V(Q) \cap V(G)|$ that is internally disjoint from $P$, and let $v_0$ be the endpoint of $Q$ in $P$. Then $\widehat{G}_\leftrightarrow[Y]$ contains three directed subpaths, $v_0 Pu_b$, $v_0 Pu_r$ and $Q$. By definition, these paths are internally disjoint. Moreover, the minimality of $|Y \cap V(G)|$ implies there is no shorter path from $v_0$ to any side of the 3-frame, otherwise we could use it to replace the current path to that side.
\end{proof}

With this structural characterization in hand, we can search for a minimum size cover of a net using the single-source shortest path algorithm for planar graphs by Henzinger et al. \cite{Henzinger1997}

\begin{alg}[Net-Alg] \ 

\noindent \textbf{Input} A connected plane graph, $G$, and a 3-frame of $G$, \\ $(W = (u_0, u_1, ..., u_n), j, k)$.  

\noindent \textbf{Idea:} Using the characterization in lemma~\ref{vine_cover}, we  search through each vertex in the face graph, finding shortest paths from a ``root'' vertex to each of the three sides of our net. By minimizing the sum of the distances given by these paths, we obtain a $(W, j, k)$-vine in $\widehat{G}$ using the fewest possible vertices from $G$.

\noindent \textbf{Initialization:} Construct the directed plane graph, $\widehat{G}_\leftrightarrow$. Let $|V(\widehat{G}_\leftrightarrow)| = s$, and give any order to the vertices, $V(\widehat{G}_\leftrightarrow) = \{v_1, ..., v_s\}$. Set $bestcover = s$ and $center = 0$.

\noindent\textbf{Iteration:}
\begin{enumerate}
	\item For each $i = 1, 2, ..., s$:
	\begin{enumerate}
		\item Run Henzinger et al.'s single-source shortest path algorithm \cite{Henzinger1997} on $\widehat{G}_\leftrightarrow$ with source $v_i$, obtaining a weighted distance, $\delta_\leftrightarrow(v_i, u)$, for each $u \in V(\widehat{G}_\leftrightarrow)$.
		\item Set
		\begin{align*}
		b(i) &= \min\{\delta_\leftrightarrow(v_i, u) : u \in V(u_0Wu_j)\} \\
		r(i) &= \min\{\delta_\leftrightarrow(v_i, u) : u \in V(u_jWu_k)\} \\
		y(i) &= \min\{\delta_\leftrightarrow(v_i, u) : u \in V(u_kWu_n)\} 
		\end{align*}
		\item Set $d(i) = b(i) + r(i) + y(i) + \mathbf{1}_G(v_i)$.
		\item Set $center= \left\{\begin{array}{ll} center & \text{ if } bestcover \leq d(i) \\
							i & \text{ if } bestcover > d(i)
							\end{array}\right.$
		\item Set $bestcover = \min\{bestcover, \ d(i)\}$.
	\end{enumerate}
	\item Run the single-source shortest path algorithm from \cite{Henzinger1997} on $\widehat{G}_\leftrightarrow$ with source, $center$, to find $P_b$, $P_r$ and $P_y$, shortest length paths for from $center$ to $V(u_0Wu_j)$, $V(u_jWu_k)$ and $V(u_kWu_n)$, respectively. Stop.
\end{enumerate}

\noindent \textbf{Output:} Define the vertex set $Y := V(P_b) \cup V(P_r) \cup V(P_y)$.
\end{alg}

\begin{thm} Given a connected plane graph $G$ with at 3-frame $\mathcal{F}$, Net-Alg runs in $O(|G|^2)$ time and outputs $Y\subseteq V(\widehat{G})$, where 
$$|Y \cap V(G)| = \emph{\ord}(\mathbf{N}(G, \mathcal{F})).$$
\end{thm}
\begin{proof} By construction, $Y$ is an $\mathcal{F}$-vine and lemma~\ref{vine_cover} implies $|Y \cap V(G)|$ is minimum among all such vines. Then theorem~\ref{min_cover} implies that $Y \cap V(G)$ is a minimum size cover of the $\mathcal{F}$-net in $G$.

In the initialization of the algorithm, we can construct $\widehat{G}$ in $O(|G|)$ time. From this graph, we can obtain $\widehat{G}_\leftrightarrow$ in $O(|G|)$ time since the number of edges in any planar graph contains at most $3(|G|-2)$ edges. In step 1 of the iteration, Henzinger's algorithm runs in $O(|G|)$ time. We run this algorithm for each vertex in $\widehat{G}$, so step 1 completes in $O(|G|^2)$ time. Step 2 runs in $O(|G|)$ time. Therefore, the running time of Net-Alg is $O(|G|^2)$.
\end{proof}

Now that we have an algorithm for finding a minimum net cover in a particular framed plane graph, we will use it to search a graph (and subgraphs) for large order nets.

\section{BT-Alg: upper and lower treewidth bounds}

We are interested not only in the order of a net in a particular framing of a plane graph, but more importantly in what nets are possible in subgraphs of that graph. Any net is hightly sensitive to the walk peripheral to the unbounded face of the embedding --- if this walk is something simple like a 3-cycle, it severely limits the complexity of the net. However, higher order nets may be lurking in the interior of this embedding.

\begin{defn} For any plane graph, $G$, let $\lambda(G)$ denote the maximum order of any net in a subgraph of $G$.
\end{defn}

\begin{lem} \label{cover_props} Let $G$ be a connected plane graph, let $W = (u_0, ..., u_n)$ be a closed walk peripheral to the unbounded face of $G$, let $\mathcal{F} = (W, j, k)$ be a 3-frame of $G$, and let $Y \subseteq V(\widehat{G})$ be an $\mathcal{F}$-vine obtained by Net-Alg. Suppose $G'$ is a connected component of $G - (Y\cap V(G))$. If $u_a, u_b \in \{u_0, ..., u_n\} \cap V(G')$ with $a \leq b$, then either $u_aWu_b$ or $(u_bWu_n, u_1Wu_a)$ is a walk in $G'$.
\end{lem}
\begin{proof} We prove the statement by contradiction. Suppose there are integers $s$ and $t$ so that $u_s \in V(u_aWu_b)$, $u_t \in V(u_bWu_n, u_1Wu_a)$ and $u_s, u_t \in Y$. $\widehat{G}[Y]$ is connected by definition, so there must be a $(u_s, u_t)$-path in $ \widehat{G}[Y]$. Since $(u_s, u_t)$ and $(u_a, u_b)$ cross in $W$, any $(u_a, u_b)$-path in $G'$ must contain a vertex in $Y$. This contradicts the hypothesis that $u_a$ and $u_b$ are contained in the same connected component of $G - (Y\cap V(G))$. 
\end{proof}

Lemma~\ref{cover_props} implies that for any connected component $G'$ of $G - (Y \cap V(G))$, a closed walk peripheral to  $G'$ can be decomposed into two internally disjoint subwalks overlapping on their endpoints. One of these subwalks is a subwalk of $W$ and the other has no interior vertices in $W$; that is, the interior vertices of this subwalk are uncolored in $G$. 

We are ready to present an algorithm for finding a high order net in a subgraph of $G$. Our algorithm is inspired by Bodlaender, Grigoriev and Koster's algorithm for finding a large square grid as a minor of a planar graph \cite{Bodlaender2008}. See Figure~\ref{fig5} for an example of its output.

\begin{alg}[BT-Alg] \label{find_net} \

\noindent\textbf{Input:} A connected planar graph, $G_0$.

\noindent\textbf{Initialization:} Use Hopcroft and Tarjan's algorithm \cite{Hopcroft1974} to obtain a planar embedding of $G_0$, then define a 3-frame $\mathcal{F}_0$ of $G_0$. Set $bestlow=-1$.

\noindent \textbf{Idea:} We create a rooted tree search by associating our framed graph with a root node. We then remove a minimum order cover of the net to separate the graph into component subgraphs, each associated with a child of the root node, and we describe a consistent way in which to define a new frame on each component. We keep track of the current greatest order of any net found in this process with $bestlow$.

\noindent\textbf{Iteration:} Input the plane graph, $G_i $, and a 3-frame, $\mathcal{F}_{i}$. Note that $i$ is a tuple. If $|G_i| < bestlow$, then let $C_i = V(G_i)$ and proceed to the next ordered node in a breadth-first search. If no nodes remain unsearched, then stop. Otherwise, $|G_i|\geq bestlow$. Then do:

\begin{enumerate}

\item Run Net-Alg % from previous section of paper
for $G_i$ to find a minimum sized cover, $C_i\subseteq V(G)$, of $\mathbf{N}(G_i, \mathcal{F}_i)$. Set $bestlow = \max\{bestlow, |C_i|\}$. 

\item Let $G_{(i,1)}, G_{(i,2)}, \ldots, G_{(i,m)}$ be the components of $G_i-C_i$. 

\item For each $q\in [m]$, we choose a 3-frame that is consistent with the colors on the 3-frame of $G_i$ as follows. We define a 3-frame $\mathcal{F}_{(i,q)}$ of $G_{(i,q)}$ as follows. Every vertex that had a color in $G_i$ retains its color in $\mathcal{F}_{(i, q)}$. Lemma~\ref{cover_props} implies the colored vertices in $\mathcal{F}_{(i,q)}$ form a subwalk of $W_i$ containing at most two colors on its vertices; let $e_1$ and $e_2$ be the endpoints of this subwalk. Note that $W_i$ may be a single vertex (in which case $e_1=e_2$) or it may be the empty walk (in which case we will let $e_1=e_2$ be an arbitrary vertex peripheral to the unbounded face of $G_{(i, q)}$). Let $S_{(i, q)}$ be the $(e_1,e_2)$-subwalk peripheral to $G_{(i,q)}$ whose interior vertices are not colored in $\mathcal{F}_i$.
\begin{enumerate}
	\item \emph{If exactly one color is missing in $G_{(i, q)}$}:
Assign every vertex of $S_{(i,q)}$ with the missing color. 
 	\item \emph{If exactly two colors are missing in $G_{(i, q)}$}: 
	Choose (arbitrarily) one of the missing colors and assign every vertex of $S_{(i,q)}$ with that color. Then choose one of $e_1$ or $e_2$ and assign to it the other missing color in addition.
	%either $e_1$ or $e_2$ with the remaining missing color.	
	\item \emph{If all three colors are missing in $G_{(i, q)}$}: 
	Choose (arbitrarily)  one of the missing colors and assign every vertex in $S_{(i, q)}$ with that color. Then choose any vertex from $S_{(i,q)}$ and assign it both of the two remaining colors in addition.
\end{enumerate} 

In each case, the three colors determine a 3-frame $\mathcal{F}_{(i, q)}$ of $G_{(i,q)}$. 

\item Recurse on child nodes $(i,1), (i,2), \ldots, (i,m)$ in a breadth-first search. 
\end{enumerate} 

\noindent\textbf{Output:} $KB=bestlow$. 
\end{alg}

\begin{figure}
\begin{center}
\begin{tikzpicture}[scale=.42]

	\node at (-1,4) {$G_0$};
	\node at (-4.5,-1) {$G_{(0,1)}$};
	\node at (8.5,-1) {$G_{(0,2)}$};
	\node at (-8,-6) {$G_{(0,1,1)}$};
	\node at (5.5,-10) {$G_{(0,1,2)}$};

	\draw [dashed] (1.732, 2) circle [radius=2.6];
	\draw [dashed] ({1.732-3}, {2-5}) circle [radius=2.6];
	\draw [dashed] ({1.732}, {2-10}) circle [radius=2.6];
	\draw [dashed] ({1.732-6}, {2-10}) circle [radius=2.6];
	\draw [dashed] ({1.732+3}, {2-5}) circle [radius=2.6];
	
	\draw ({1.732 - 1.34}, -8 + 2.23) -- ({1.732 - 3 + 1.34}, {2 - 5 - 2.23});
	\draw ({1.732-6 + 1.34}, -8 + 2.23) -- ({1.732 - 3 - 1.34}, {2 - 5 - 2.23});
	\draw ({1.732-3 + 1.34}, -3 + 2.23) -- ({1.732 - 1.34}, {2 - 2.23});
	\draw ({1.732+3 - 1.34}, -3 + 2.23) -- ({1.732 + 1.34}, {2 - 2.23});

	\draw [line width=0.5mm, blue] (1.732, 0) -- (0, 1) -- (0, 3);
	\draw [line width=0.5mm, red] (0, 3) -- (1.732, 4) -- (3.464, 3);
	\draw [line width=0.5mm, yellow] (3.464, 3) -- (3.464, 1) -- (1.732, 0);
	\draw (0, 1) -- (1.732, 2) -- (3.464, 1);
	\draw (0.866, 0.5) -- (2.598, 1.5) -- (2.598, 3.5);
	\draw (0.866, 3.5) -- (0.866, 1.5) -- (2.598, 0.5);
	\draw (1.732, 4) -- (1.732, 2);
	\draw (0, 2) -- (1.732, 3) -- (3.464, 2);
	
	\draw [fill, blue] (0, 1) circle [radius=0.1];
	\draw [fill, blue] (0, 2) circle [radius=0.1];
	\draw [fill, violet] (0, 3) circle [radius=0.1];
	\draw [fill, blue] (0.866, 0.5) circle [radius=0.1];
	\draw [fill] (0.866, 1.5) circle [radius=0.1];
	\draw [fill] (0.866, 2.5) circle [radius=0.1];
	\draw [fill, red] (0.866, 3.5) circle [radius=0.1];
	\draw [fill, green] (1.732, 0) circle [radius=0.1];
	\draw [fill] (1.732, 1) circle [radius=0.1];
	\draw [fill] (1.732, 2) circle [radius=0.1];
	\draw [fill] (1.732, 3) circle [radius=0.1];
	\draw [fill, red] (1.732, 4) circle [radius=0.1];
	\draw [fill, yellow] (2.598, 0.5) circle [radius=0.1];
	\draw [fill] (2.598, 1.5) circle [radius=0.1];
	\draw [fill] (2.598, 2.5) circle [radius=0.1];
	\draw [fill, red] (2.598, 3.5) circle [radius=0.1];
	\draw [fill, yellow] (3.464, 1) circle [radius=0.1];
	\draw [fill, yellow] (3.464, 2) circle [radius=0.1];
	\draw [fill, orange] (3.464, 3) circle [radius=0.1];
	
	\draw [line width=.75mm, dotted, gray] (1.732, 0) -- (1.732, 2) -- (2.598, 2.5) -- (2.598, 3.5);

	\draw [line width=0.5mm, yellow] ({3.464+3}, -2) -- ({3.464+3}, -4) -- ({2.598+3}, -4.5);
	\draw ({3.464+3}, -4) -- ({2.598+3}, -3.5);
	
	\draw [fill, yellow] ({2.598+3}, -4.5) circle [radius=0.1];
	\draw [fill] ({2.598+3}, -3.5) circle [radius=0.1];
	\draw [fill, yellow] ({3.464+3}, -4) circle [radius=0.1];
	\draw [fill, yellow] ({3.464+3}, -3) circle [radius=0.1];
	\draw [fill, orange] ({3.464+3}, -2) circle [radius=0.1];

	\draw [line width=0.5mm, green] ({0.866-3}, -4.5) -- (-3, -4);
	\draw [line width=0.5mm, blue] (-3, -4) -- (-3, -2);
	\draw [line width=0.5mm, red] ({0-3}, -2) -- ({1.732-3}, -1);
	\draw [line width=0.5mm, yellow] (-3, -4) -- ({0.886-3}, -3.5);
	\draw [line width=0.5mm, yellow] ({0.866-3}, -2.5) -- ({0.866-3}, -3.5);
	\draw ({0.866-3}, -1.5) -- ({0.866-3}, -2.5);
	\draw [line width=0.5mm, yellow] ({1.732-3}, -1) -- ({1.732-3}, -2);
	\draw [line width=0.5mm, yellow] ({0.866-3}, -2.5) -- ({1.732-3}, -2);
	\draw (-3, -3) -- ({0.866-3}, -2.5);
	
	\draw [fill, green] (-3, -4) circle [radius=0.1];
	\draw [fill, blue] (-3, -3) circle [radius=0.1];
	\draw [fill, violet] (-3, -2) circle [radius=0.1];
	\draw [fill, green] ({0.866-3}, -4.5) circle [radius=0.1];
	\draw [fill, yellow] ({0.866-3}, -3.5) circle [radius=0.1];
	\draw [fill, yellow] ({0.866-3}, -2.5) circle [radius=0.1];
	\draw [fill, red] ({0.866-3}, -1.5) circle [radius=0.1];
	\draw [fill, yellow] ({1.732-3}, -2) circle [radius=0.1];
	\draw [fill, orange] ({1.732-3}, -1) circle [radius=0.1];
	
	\draw [line width=.75mm, dotted, gray] (-3, -2) -- (-2, -2.5);
	
	\draw [line width=0.5mm, green] (-6, -9) -- ({0.866-6}, {-4.5-5});
	\draw [line width=0.5mm, blue] (-6, -9) -- (-6,-8);
	\draw [line width=0.5mm, yellow] (-6, -9) -- ({0.866-6}, {-3.5-5});
	
	\draw [fill, green] (-6, -9) circle [radius=0.1];
	\draw [fill, blue] (-6, -8) circle [radius=0.1];
	\draw [fill, green] ({0.866-6}, {-4.5-5}) circle [radius=0.1];
	\draw [fill, yellow] ({0.866-6}, {-3.5-5}) circle [radius=0.1];
	
	\draw [line width=0.5mm, yellow] ({1.732}, {-1-5}) -- ({1.732}, {-2-5});
	\draw [line width=0.5mm, red] ({1.732}, {-1-5}) -- ({0.866}, {-1.5-5});
	
	\draw [fill, red] ({0.866}, {-1.5-5}) circle [radius=0.1];
	\draw [fill, yellow] ({1.732}, {-2-5}) circle [radius=0.1];
	\draw [fill, orange] ({1.732}, {-1-5}) circle [radius=0.1];

\end{tikzpicture}
\caption{An example of BT-Alg resulting in $KB=5$.} \label{fig5}
\end{center}
\end{figure}
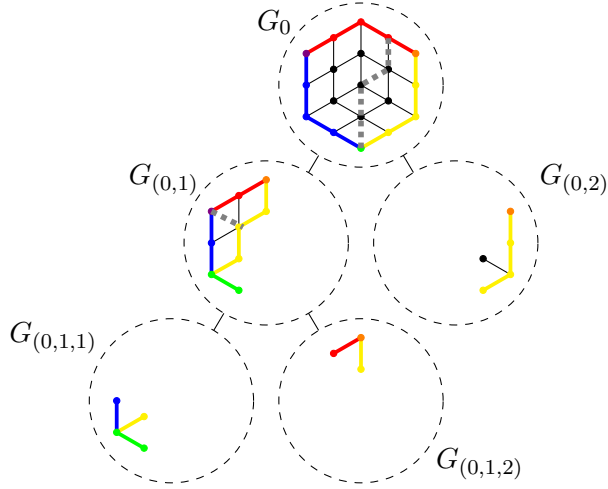

We were careful in how we defined the 3-frames for child nodes of the search tree so that the net covers obtained at each step of the algorithm have the following nice property as separating sets in $G$.

\begin{lem} \label{3_colors} Any subgraph associated to a node in the rooted search tree obtained in BT-Alg is incident with at most three cover sets found in a previous iteration of step 1.
\end{lem}
\begin{proof} Let $i$ be a node in the rooted search tree obtained by BT-Alg. If $u \in V(G_i)$ is adjacent to some cover set, say $C_{i'}$, associated to a previous node $i'$ in the search tree, then $C_{i'}$ separated $u$ from at least one color of $\mathcal{F}_{i'}$ in $G_{i'}$, and $u$ was assigned one of those missing colors, say $blue$. No later iteration between $i$ and $i'$ can separate a connected component containing $G_i$ from the $blue$ vertices. That is, $C_{i'}$ is the unique cover set incident with $G_i$ and separating $V(G_i)$ from the $blue$ side of the frame. Since there are only three colors in a 3-frame, there can only be three such cover sets incident with $G_i$.
\end{proof}

%\subsection{Lower bound on treewidth}

\begin{thm} \label{BT-Alg}  For any connected planar graph $G$, the constant $KB$ output by BT-Alg satisfies $\lambda(G)/4\leq KB\leq \lambda(G)$. 
\end{thm}

\begin{proof} Since BT-Alg finds a net of order $KB$ in a subgraph of $G$, we know that $KB \leq \lambda(G)$. Let $\mathbf{M}$ be a net of order $\lambda(G)$ in some subgraph of $G$. Lemma~\ref{net_bramble} implies that every vine in $\mathbf{M}$ covers $\mathbf{M}$, so each vine contains at least $\lambda(G)$ vertices. Consider a node $i$ in our search tree such that $V(G_i)$ contains a vine, say $Y \in \mathbf{M}$, and no child of $G_i$ contains a vine in $\mathbf{M}$. Such a node must exist in our search tree since the graph associated to each leaf has fewer than $\lambda(G)$ vertices.

By lemma~\ref{3_colors}, $G_i$ is incident with at most three net covers obtained in BT-Alg. Thus, every vine in $\mathbf{M}$ which is not contained in $G_i$ intersects at least one of these three net covers nontrivially. Furthermore, every vine in $\mathbf{M}$ that is contained in $G_i$ must intersect $C_i$. Now we can cover $\mathbf{M}$ with four covers obtained in the algorithm, at most three incident with $G_i$ plus $C_i$. By construction, each of these covers has size at most $KB$, so $\lambda(G) = \ord(\mathbf{M}) \leq 4KB$.
\end{proof}

%\subsection{Upper bound on treewidth}

Inspired by \cite{Grigoriev2011}, we use BT-Alg to construct a tree decomposition of a planar graph. This will give an upper bound on treewidth.

\begin{thm} \label{upper} For any connected planar graph $G$, the constant $KB$ output by BT-Alg satisfies $TW(G) \leq 4KB - 1$.
\end{thm}
\begin{proof} We define a tree decomposition of $G$ as follows. Let $T$ be the tree constructed in the rooted tree search in BT-Alg.  For each node $i \in V(T)$, define 
	$$\beta(i) = \ C_i \cup \{u \in V(G)\backslash V(G_i) : N_G(u) \cap V(G_i) \neq \varnothing\}.$$
That is, each vertex set $\beta(i)$ contains the net cover, $C_i$, plus all vertices outside of $G_i$ which have a neighbor in $G_i$. We need to confirm that $(T, \beta)$ possesses the three properties of a tree decomposition. 

(1): In the algorithm, for each $u \in V(G)$, there is some $i \in V(T)$ such that $u \in C_i$, so $u \in \beta(i)$. 

(2): Suppose $uv \in E(G)$, $u \in C_i$ and $v \in C_{i'}$. If $i = i'$ or both $i \not= i'$ and $v \in V(G) \backslash V(G_i)$, then $u, v \in \beta(i)$. Otherwise $v \in V(G_i) \backslash C_i$, implying $G_{i'}$ is a subgraph of some connected component of $G_i - C_i$. By definition $u, v \in \beta(i')$.

(3): The third property is equivalent to showing that for any $u \in V(G)$, $\beta^{-1}(u)$ is connected. For any $u \in V(G)$, 
	$$\beta^{-1}(u) = \{i \in V(T) : u \in C_i \text{ or } (u \notin V(G_i)  \ \text{and} \ \exists v \in N_G(u) \cap V(G_i) ) \}.$$
If $u \in C_j$, then $T[\beta^{-1}(u)]$ will be a subtree of $T$ rooted at $j$.

We have shown that $(T, \beta)$ is a tree decomposition of $G$. For any $i \in V(T)$, lemma~\ref{3_colors} implies $\beta(i)$ intersects at most four net covers defined in the algorithm. Therefore, $|\beta(i)| \leq 4KB$. That is, the width of $(T, \beta)$ is at most $4KB - 1$.
\end{proof}

Combining Theorems~\ref{BT-Alg} and \ref{upper}, we get the following result on lower and upper bounds on the treewidth of a planar graph:

\begin{thm} \label{lower-and-upper} Let $G$ be a planar graph. Then BT-Alg computes $KB$ in $O(|G|^3)$ time, and 
$$\frac{\lambda(G)}{4} \leq KB \leq BN (G) \leq 4KB \leq 4\lambda(G).$$

%Let $G$ be a planar graph. Then 
%$$\frac{\lambda(G)}{4} \leq KB \leq BN (G) \leq 4KB \leq 4\lambda(G),$$
%where BT-Alg 
%algorithm~\ref{find_net} 
%computes $KB$ in $O(|G|^3)$ time.
\end{thm}
\begin{proof}
	We run BT-Alg on each connected component of $G$. Hopcroft and Tarjan's algorithm generates a planar embedding in $O(|G|)$ time. In the iteration, Net-Alg runs in $O(|G|^2)$ time. Moreover, the nodes in the search tree are in bijection with disjoint, nonempty cut-sets from $G$, so this iteration runs at most $|G|$ times. Therefore, the rooted tree search completes in $O(|G|^3)$ time. The upper and lower bounds on bramble number follow directly from theorem~\ref{upper} and theorem~\ref{BT-Alg}, respectively. 
\end{proof}

\section{Correcting a lower bound with square grids}

In this section, we explain the difficulties with \cite{Bodlaender2008}, in particular, the proof of Theorem~3.2(\cite{Bodlaender2008}) and the algorithm $\mathcal A_2$(\cite{Bodlaender2008}). Algorithm $\mathcal A_2$(\cite{Bodlaender2008}) defines four sides to a plane graph and it determines a tree of north-south or east-west cuts, and it inspired us to write BT-Alg. Theorem~3.2(\cite{Bodlaender2008}) finds lower bound $LB_2$ using Algorithm $\mathcal A_2$(\cite{Bodlaender2008}). It inspired us to prove Theorem~5.5. 
In the examples below, we follow their notation, using $\mathcal N$, $\mathcal E$, $\mathcal S$, $\mathcal W$ as the dummy vertices adjacent to the paths $North$, $East$, $South$, $West$.

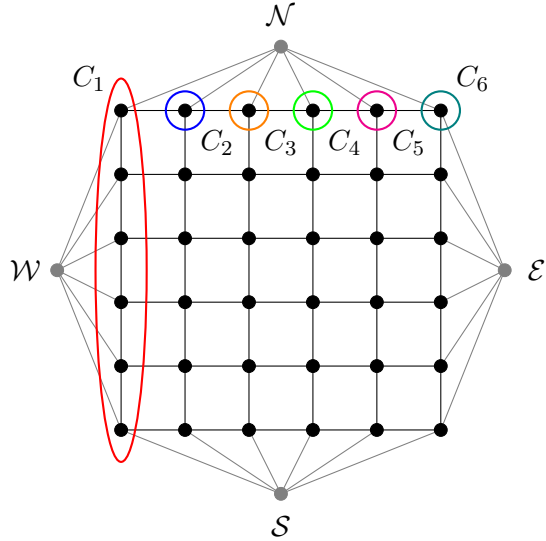
\begin{figure}[h] \label{Problem-2}
\begin{center}
\begin{tikzpicture}[scale=.85]

	\draw [gray] (-1,2.5) -- (0,0);
	\draw [gray] (-1,2.5) -- (0,1);
	\draw [gray] (-1,2.5) -- (0,2);
	\draw [gray] (-1,2.5) -- (0,3);
	\draw [gray] (-1,2.5) -- (0,4);
	\draw [gray] (-1,2.5) -- (0,5);
	
	\draw [gray] (6,2.5) -- (5,0);
	\draw [gray] (6,2.5) -- (5,1);
	\draw [gray] (6,2.5) -- (5,2);
	\draw [gray] (6,2.5) -- (5,3);
	\draw [gray] (6,2.5) -- (5,4);
	\draw [gray] (6,2.5) -- (5,5);
	
	\draw [gray] (2.5,-1) -- (0,0);
	\draw [gray] (2.5,-1) -- (1,0);
	\draw [gray] (2.5,-1) -- (2,0);
	\draw [gray] (2.5,-1) -- (3,0);
	\draw [gray] (2.5,-1) -- (4,0);
	\draw [gray] (2.5,-1) -- (5,0);
	
	\draw [gray] (2.5,6) -- (0,5);
	\draw [gray] (2.5,6) -- (1,5);
	\draw [gray] (2.5,6) -- (2,5);
	\draw [gray] (2.5,6) -- (3,5);
	\draw [gray] (2.5,6) -- (4,5);
	\draw [gray] (2.5,6) -- (5,5);
	
	\draw (0,0) -- (0,5) -- (5,5) -- (5,0) -- cycle;
	\draw (0,1) -- (5,1);
	\draw (0,2) -- (5,2);
	\draw (0,3) -- (5,3);
	\draw (0,4) -- (5,4);
	\draw (1,0) -- (1,5);
	\draw (2,0) -- (2,5);
	\draw (3,0) -- (3,5);
	\draw (4,0) -- (4,5);
	
	\draw [fill] (0, 0) circle [radius=0.1];
	\draw [fill] (1, 0) circle [radius=0.1];
	\draw [fill] (2, 0) circle [radius=0.1];
	\draw [fill] (3, 0) circle [radius=0.1];
	\draw [fill] (4, 0) circle [radius=0.1];
	\draw [fill] (5, 0) circle [radius=0.1];
	\draw [fill] (0, 1) circle [radius=0.1];
	\draw [fill] (1, 1) circle [radius=0.1];
	\draw [fill] (2, 1) circle [radius=0.1];
	\draw [fill] (3, 1) circle [radius=0.1];
	\draw [fill] (4, 1) circle [radius=0.1];
	\draw [fill] (5, 1) circle [radius=0.1];
	\draw [fill] (0, 2) circle [radius=0.1];
	\draw [fill] (1, 2) circle [radius=0.1];
	\draw [fill] (2, 2) circle [radius=0.1];
	\draw [fill] (3, 2) circle [radius=0.1];
	\draw [fill] (4, 2) circle [radius=0.1];
	\draw [fill] (5, 2) circle [radius=0.1];
	\draw [fill] (0, 3) circle [radius=0.1];
	\draw [fill] (1, 3) circle [radius=0.1];
	\draw [fill] (2, 3) circle [radius=0.1];
	\draw [fill] (3, 3) circle [radius=0.1];
	\draw [fill] (4, 3) circle [radius=0.1];
	\draw [fill] (5, 3) circle [radius=0.1];
	\draw [fill] (0, 4) circle [radius=0.1];
	\draw [fill] (1, 4) circle [radius=0.1];
	\draw [fill] (2, 4) circle [radius=0.1];
	\draw [fill] (3, 4) circle [radius=0.1];
	\draw [fill] (4, 4) circle [radius=0.1];
	\draw [fill] (5, 4) circle [radius=0.1];
	\draw [fill] (0, 5) circle [radius=0.1];
	\draw [fill] (1, 5) circle [radius=0.1];
	\draw [fill] (2, 5) circle [radius=0.1];
	\draw [fill] (3, 5) circle [radius=0.1];
	\draw [fill] (4, 5) circle [radius=0.1];
	\draw [fill] (5, 5) circle [radius=0.1];
	
	\draw [thick, red] (0,2.5) ellipse (0.4 and 3);
	\draw [thick, blue] (1, 5) circle [radius=0.3];
	\draw [thick, orange] (2, 5) circle [radius=0.3];
	\draw [thick, green] (3, 5) circle [radius=0.3];
	\draw [thick, magenta] (4, 5) circle [radius=0.3];
	\draw [thick, teal] (5, 5) circle [radius=0.3];
	
	\node at (-0.5,5.5) {$C_1$};
	\node at (1.5,4.5) {$C_2$};
	\node at (2.5,4.5) {$C_3$};
	\node at (3.5,4.5) {$C_4$};
	\node at (4.5,4.5) {$C_5$};
	\node at (5.5,5.5) {$C_6$};
	
	\draw [fill, gray] (-1, 2.5) circle [radius=0.1];
	\draw [fill, gray] (2.5,-1) circle [radius=0.1];
	\draw [fill, gray] (6, 2.5) circle [radius=0.1];
	\draw [fill, gray] (2.5,6) circle [radius=0.1];
	
	\node at (2.5,6.5) {$\mathcal{N}$};
	\node at (2.5,-1.5) {$\mathcal{S}$};
	\node at (-1.5,2.5) {$\mathcal{W}$};
	\node at (6.5,2.5) {$\mathcal{E}$};
\end{tikzpicture}
\end{center}
\caption{After six iterations of $\mathcal{A}_2$(\cite{Bodlaender2008}), the only component subgraph is incident with all six cuts. }
\end{figure}

The first step of $\mathcal A_2$(\cite{Bodlaender2008}) assigns roughly equal numbers of vertices to each of $North$, $East$, $South$ and $West$. We note that throughout the algorithmic procedure, it is impossible to maintain four paths of roughly the same length. For instance, if a cut dramatically increases the number of peripheral vertices, then it is possible that one side will contain much less than one-quarter of the vertices in the new periphery.  Thus, we cannot assume that the paths will be of approximately equal length or that this constitutes an assignment of vertices to the new periphery. 

The problem with $\mathcal A_2$(\cite{Bodlaender2008}) is that it is not specific in its assignment of new peripheral vertices to $North$, $East$, $South$ or $West$ at each iteration. When making the periphery assignments, there are choices that will not yield the desired result in the proof of Theorem~3.2(\cite{Bodlaender2008}). Figure~\ref{Problem-2} describes such a situation. Our example is a $6\times 6 $ grid with $\mathcal N$, $\mathcal E$, $\mathcal S$, $\mathcal W$  each adjacent to a side of the grid. We choose $C_1$ to be the leftmost column that separates $\mathcal E $ and $\mathcal W$. The rule we use to assign vertices to the periphery is that we extend $South$ to meet $North$, and $West$ is atomic. Then $C_i$,  $2\leq i\leq 6$, is the singleton vertex in $West$. This yields a connected component that is adjacent to six separations, where the expected number was four.

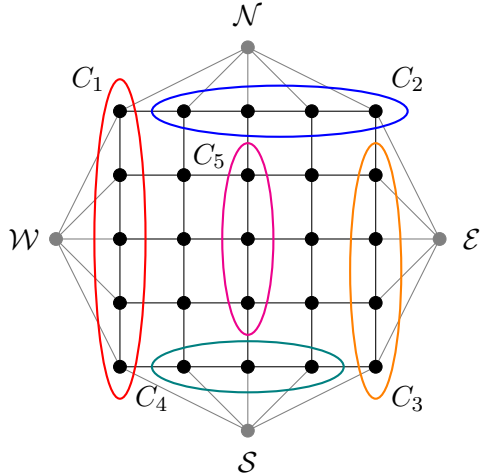
\begin{figure}[h]
\begin{center}
\begin{tikzpicture}[scale=.85]
	\draw [fill, gray] (-1, 2) circle [radius=0.1];
	\draw [fill, gray] (2,-1) circle [radius=0.1];
	\draw [fill, gray] (5, 2) circle [radius=0.1];
	\draw [fill, gray] (2,5) circle [radius=0.1];
	
	\draw [gray] (-1,2) -- (0,0);
	\draw [gray] (-1,2) -- (0,1);
	\draw [gray] (-1,2) -- (0,2);
	\draw [gray] (-1,2) -- (0,3);
	\draw [gray] (-1,2) -- (0,4);
	
	\draw [gray] (5,2) -- (4,0);
	\draw [gray] (5,2) -- (4,1);
	\draw [gray] (5,2) -- (4,2);
	\draw [gray] (5,2) -- (4,3);
	\draw [gray] (5,2) -- (4,4);
	
	\draw [gray] (2,-1) -- (0,0);
	\draw [gray] (2,-1) -- (1,0);
	\draw [gray] (2,-1) -- (2,0);
	\draw [gray] (2,-1) -- (3,0);
	\draw [gray] (2,-1) -- (4,0);
	
	\draw [gray] (2,5) -- (0,4);
	\draw [gray] (2,5) -- (1,4);
	\draw [gray] (2,5) -- (2,4);
	\draw [gray] (2,5) -- (3,4);
	\draw [gray] (2,5) -- (4,4);
	
	\node at (2,5.5) {$\mathcal{N}$};
	\node at (2,-1.5) {$\mathcal{S}$};
	\node at (-1.5,2) {$\mathcal{W}$};
	\node at (5.5,2) {$\mathcal{E}$};

	\draw (0,0) -- (0,4) -- (4,4) -- (4,0) -- cycle;
	\draw (0,1) -- (4,1);
	\draw (0,2) -- (4,2);
	\draw (0,3) -- (4,3);
	\draw (1,0) -- (1,4);
	\draw (2,0) -- (2,4);
	\draw (3,0) -- (3,4);
	
	\draw [fill] (0, 0) circle [radius=0.1];
	\draw [fill] (1, 0) circle [radius=0.1];
	\draw [fill] (2, 0) circle [radius=0.1];
	\draw [fill] (3, 0) circle [radius=0.1];
	\draw [fill] (4, 0) circle [radius=0.1];
	\draw [fill] (0, 1) circle [radius=0.1];
	\draw [fill] (1, 1) circle [radius=0.1];
	\draw [fill] (2, 1) circle [radius=0.1];
	\draw [fill] (3, 1) circle [radius=0.1];
	\draw [fill] (4, 1) circle [radius=0.1];
	\draw [fill] (0, 2) circle [radius=0.1];
	\draw [fill] (1, 2) circle [radius=0.1];
	\draw [fill] (2, 2) circle [radius=0.1];
	\draw [fill] (3, 2) circle [radius=0.1];
	\draw [fill] (4, 2) circle [radius=0.1];
	\draw [fill] (0, 3) circle [radius=0.1];
	\draw [fill] (1, 3) circle [radius=0.1];
	\draw [fill] (2, 3) circle [radius=0.1];
	\draw [fill] (3, 3) circle [radius=0.1];
	\draw [fill] (4, 3) circle [radius=0.1];
	\draw [fill] (0, 4) circle [radius=0.1];
	\draw [fill] (1, 4) circle [radius=0.1];
	\draw [fill] (2, 4) circle [radius=0.1];
	\draw [fill] (3, 4) circle [radius=0.1];
	\draw [fill] (4, 4) circle [radius=0.1];
	
	\draw [thick, red] (0,2) ellipse (0.4 and 2.5);
	\draw [thick, blue] (2.5,4) ellipse (2 and 0.4);
	\draw [thick, orange] (4,1.5) ellipse (0.4 and 2);
	\draw [thick, teal] (2,0) ellipse (1.5 and 0.4);
	\draw [thick, magenta] (2,2) ellipse (0.4 and 1.5);
	
	\node at (-0.5,4.5) {$C_1$};
	\node at (4.5,4.5) {$C_2$};
	\node at (4.5,-0.5) {$C_3$};
	\node at (0.5,-0.5) {$C_4$};
	\node at (1.4,3.3) {$C_5$};
	
\end{tikzpicture}
\end{center}
\caption{After four iterations of $\mathcal{A}_2$(\cite{Bodlaender2008}), there is only one component subgraph. The two subgraphs associated to the child nodes are defined by five cuts in the algorithm.} \label{Prob3}
\end{figure} 

Next we consider $\mathcal A_2$(\cite{Bodlaender2008}) assuming that the assignment of vertices to the periphery occurs as in BT-Alg. We will show that there is an error in the proof of Theorem~3.2(\cite{Bodlaender2008}) and that the proof shows only $LB_2\geq s(G)/5$, where $s(G)$ is the size of largest square grid minor of $G$. In the proof the authors construct a rooted-search-tree where each node $i$ is associated to a subgraph $G_i$ of $G$. %along with four peripheral paths. 
Each $G_i$ is determined by at most two $\mathcal N-\mathcal S$ vertex cuts and at most two $\mathcal E-\mathcal W$ vertex cuts, hence at most four vertex cuts. The children of $G_i$ are the connected components of $G_i - C_i$, where $C_i$ is an additional vertex cut, either $\mathcal N-\mathcal S$ or $\mathcal E-\mathcal W$. The authors assume that at most four vertex cuts are enough to separate all child nodes of $G_i$. They use this claim in the very last sentence of the proof of Theorem~3.2(\cite{Bodlaender2008}), to show some one of these cuts must have size at least $s(G)/4$.

However, it may take five vertex cuts to determine, simultaneously, all subgraphs associated to child nodes of $i$. Although each of these children are determined by at most four cuts, to simultaneously separate all of them, we must include the four that separate $G_i$ and the cut that divides $G_i$ into its children. 
Figure~\ref{Prob3} is an example of when five cuts are needed. If we implement algorithm $\mathcal{A}_2$(\cite{Bodlaender2008})  on the $5 \times 5$ grid, then $\mathcal{A}_2$(\cite{Bodlaender2008}) could produce the sequence of five cut sets pictured, $C_1$, $C_2$, $C_3$, $C_4$ and $C_5$. After four iterations of the algorithm, there is only one component subgraph associated to a single node in the rooted-search-tree. Then removing $C_5$ produces two component subgraphs, each of which is adjacent to only four cuts in the algorithm. However, all five cuts need to be removed from the original grid in order to obtain both component subgraphs.

Thus, the bound proved in Theorem~3.2(\cite{Bodlaender2008}) shows $LB_2\geq s(G)/5$, not $s(G)/4$. 

\section{Concluding remarks}

We have shown that nets are a natural generalization of high order brambles found in square grids. Moreover, using only three sides in their definition makes nets a more natural candidate for a high order obstruction to treewidth than the four-sided square grid minors. In practice, because a square grid minor has more structure than a net, it is harder to find. Thus we expect our algorithm to perform better than $\mathcal A_2$(\cite{Bodlaender2008}). It would be interesting to experimentally test BT-Alg for running time and efficiency on small graphs to see if this holds. 

In the experimental results in \cite{Bodlaender2008} $LB_2$ is usually at least half of $BN(G)$, so typically $LB_2$ is much larger than the theoretical lower bound. We expect that in practice, $KB$ will be much greater than $\lambda(G)/4$. We expect only very specific constructions to achieve this theoretical lower bound.

\bibliographystyle{plain}

\bibliography{netbib}

\begin{thebibliography}{10}

\bibitem{Bellenbaum2002}
Patrick Bellenbaum and Reinhard Diestel.
\newblock Two short proofs concerning tree-decompositions.
\newblock {\em Combin. Probab. Comput.}, 11(6):541--547, 2002.

\bibitem{Birmele2007}
E.~Birmel\'e, J.~A. Bondy, and B.~A. Reed.
\newblock Brambles, prisms and grids.
\newblock In {\em Graph theory in {P}aris}, Trends Math., pages 37--44.
  Birkh\"auser, Basel, 2007.

\bibitem{Bodlaender1998}
Hans~L. Bodlaender.
\newblock A partial {$k$}-arboretum of graphs with bounded treewidth.
\newblock {\em Theoret. Comput. Sci.}, 209(1-2):1--45, 1998.

\bibitem{Bodlaender2008}
Hans~L. Bodlaender, Alexander Grigoriev, and Arie M. C.~A. Koster.
\newblock Treewidth lower bounds with brambles.
\newblock {\em Algorithmica}, 51(1):81--98, 2008.

\bibitem{Chekuri2016}
Chandra Chekuri and Julia Chuzhoy.
\newblock Polynomial bounds for the gird-minor theorem.
\newblock {\em J. ACM}, 63(5):Art. 40, 65, 2016.

\bibitem{Dreyfus1971}
S.~E. Dreyfus and R.~A. Wagner.
\newblock The {S}teiner problem in graphs.
\newblock {\em Networks}, 1:195--207, 1971/72.

\bibitem{Grigoriev2011}
Alexander Grigoriev.
\newblock Tree-width and large grid minors in planar graphs.
\newblock {\em Discrete Math. Theor. Comput. Sci.}, 13(1):13--20, 2011.

\bibitem{Henzinger1997}
Monika~R. Henzinger, Philip Klein, Satish Rao, and Sairam Subramanian.
\newblock Faster shortest-path algorithms for planar graphs.
\newblock {\em J. Comput. System Sci.}, 55(1, part 1):3--23, 1997.
\newblock 26th Annual ACM Symposium on the Theory of Computing (STOC '94)
  (Montreal, PQ, 1994).

\bibitem{Reed1997}
B.~A. Reed.
\newblock Tree width and tangles: a new connectivity measure and some
  applications.
\newblock In {\em Surveys in combinatorics, 1997 ({L}ondon)}, volume 241 of
  {\em London Math. Soc. Lecture Note Ser.}, pages 87--162. Cambridge Univ.
  Press, Cambridge, 1997.

\bibitem{Robertson1986a}
Neil Robertson and P.~D. Seymour.
\newblock Graph minors. {II}. {A}lgorithmic aspects of tree-width.
\newblock {\em J. Algorithms}, 7(3):309--322, 1986.

\bibitem{Robertson1986}
Neil Robertson and P.~D. Seymour.
\newblock Graph minors. {VI}. {D}isjoint paths across a disc.
\newblock {\em J. Combin. Theory Ser. B}, 41(1):115--138, 1986.

\bibitem{Robertson2004}
Neil Robertson and P.~D. Seymour.
\newblock Graph minors. {XX}. {W}agner's conjecture.
\newblock {\em J. Combin. Theory Ser. B}, 92(2):325--357, 2004.

\bibitem{Seymour1993}
P.~D. Seymour and Robin Thomas.
\newblock Graph searching and a min-max theorem for tree-width.
\newblock {\em J. Combin. Theory Ser. B}, 58(1):22--33, 1993.

\bibitem{Smith2015}
Brett~C. Smith.
\newblock {\em On Minimality of Planar Graphs with Respect to Treewidth}.
\newblock PhD thesis, Wesleyan University, 45 Wyllys Ave, Middletown, CT 06459,
  May 2015.

\bibitem{West1996}
Douglas~B. West.
\newblock {\em Introduction to graph theory}.
\newblock Prentice Hall, Inc., Upper Saddle River, NJ, 1996.

\end{thebibliography}

\end{document}